\documentclass{amsart}
\usepackage{amsmath,amssymb,amscd,amsthm}
\usepackage[pdftex]{graphics}

\newcommand{\Z}{\mathbb Z}

\newcommand{\Am}{\operatorname{Am}}

\newcommand{\Cl}{\operatorname{Cl}}
\newcommand{\Gal}{\operatorname{Gal}}

\newcommand{\st}{{\operatorname{st}}}

\newcommand{\coker}{\operatorname{coker}\,}

\newcommand{\fin}{\operatorname{fin}}
\newcommand{\cO}{{\mathfrak O}}

\newcommand{\fa}{{\mathfrak a}}

\newcommand{\fb}{{\mathfrak b}}

\newcommand{\fp}{\mathfrak p}
\newcommand{\fP}{{\mathfrak P}}

\newcommand{\lra}{\longrightarrow}
\newcommand{\la}{\langle}
\newcommand{\ra}{\rangle}
\newcommand{\tH}{\widetilde{H}}
\newcommand{\tD}{\widetilde{D}}

\newcounter{Tc}

\newtheorem{thm}[Tc]{Theorem}
\newtheorem{prop}{Proposition}
\newtheorem{lem}{Lemma}

\title{The Ambiguous Class Number Formula Revisited}
\author{Franz Lemmermeyer}
\begin{document}
\begin{abstract}
We will give a simple proof of the ambiguous class number formula.
\end{abstract}
\maketitle

\pagestyle{myheadings}
\markboth{Ambiguous Class Numbers}{\today \hfil Franz Lemmermeyer}

\section*{Introduction}
Let $K$ be a number field, $L/K$ a cyclic extension of prime 
degree $\ell$, and $\sigma$ a generator of its Galois group
$G = \Gal(L/K)$. We say that an ideal class $[\fa] \in \Cl(L)$ is 
\begin{itemize}
\item {\em ambiguous} if $[\fa]^\sigma = [\fa]$, that is, if there exists an 
      element $\alpha \in L^\times$ such that $\fa^{\sigma-1} = (\alpha)$, or 
      more briefly, if $\fa^{\sigma-1} \in H_L$. 
\item {\em strongly ambiguous} if $\fa^{\sigma-1} = (1)$.
\end{itemize}
The group $\Am(L/K)$ of ambiguous ideal classes is
the subgroup of the class group $\Cl(L)$ consisting of ambiguous ideal 
classes. Its subgroup of strongly ambiguous ideal classes is denoted by
$\Am_\st(L/K)$.
 
The following ``ambiguous class number formula'' is well known: it 
gives the number of (strongly) ambiguous ideal classes in terms of 
the class number $h(K)$ of the base field, the number $t$ of ramified 
primes (including those at infinity), and  the index of the units that 
are norms (of integers and units, respectively) inside the unit group 
$E_K$ of $K$:

\begin{thm}\label{main}
The number of ambiguous ideal classes is given by 
\begin{align}
\label{ham}
 \# \Am(L/K) & = h(K) \cdot \frac{\ell^{t-1}}{(E_K:E_K \cap N L^\times)}, \\
\label{hams}
 \# \Am_\st(L/K) & = h(K) \cdot \frac{\ell^{t-1}}{(E_K:N E_L)},
\end{align}
where $E_K$ is the unit group of $K$, and $E_K \cap N L^\times$ the
subgroup of units that are norms of elements of $L$.
\end{thm}

The units in $E_K \cap N L^\times$, i.e., units in $K$ that are norms 
of elements in $L$, coincide with the units that are local norms 
everywhere (i.e., norms in every completion of $L/K$). This means
that the index $(E_K: E_K \cap N L^\times)$ can be computed in a purely
local way.

Since the classical approach to class field theory has gone out
of fashion, there are almost no modern books on class field theory
in which the ambiguous class number formula is proved; exceptions 
are Lang \cite{Lang} and Gras \cite{Gras}. For this reason we would 
like to present a modern and yet very elementary way of deriving the 
ambiguous class number formula from a result that can be found in any 
introduction to class field theory: the calculation of the Herbrand 
index of the unit group.

Our approach is close to that used in Lang \cite{Lang}, but does not 
use cohomology of groups. Readers familiar with cohomology will easily
recognize some of the exact sequences below as pieces of the long
exact cohomology sequence. For example, (\ref{ExAst}) is the beginning
of the cohomology sequence attached to 
$$ \begin{CD}
    1 @>>> H_L @>>> D_L @>>> \Cl(L) @>>> 1, \end{CD} $$
which is essentially the definition of the class group. In fact, taking 
fixed modules gives
$$ \begin{CD}
    1 @>>> H_L^G @>>> D_L^G @>>> \Am(L/K), \end{CD} $$
and we obtain (\ref{ExAst}) below by replacing $\Am(L/K)$ with the image
of $D_L^G$, which is $\Am_\st(L/K)$. We do not gain any additional insight,
however, by interpreting the almost obvious exact sequence (\ref{ExAst}),
which essentially is the definition of the strongly ambiguous ideal class
group, as a part of the long cohomology sequence.

We take this opportunity to invite the readers to look at a few
articles that provide an excellent background to our computations,
or show how to do everything in cohomological terms. In particular
we recommend Martinet \cite{Mtours} concerning discriminant bounds
and the cohomological approach to investigations of class groups 
given by Cornell \& Rosen \cite{CRCoh}. 

\section{Ambiguous and Strongly Ambiguous Ideals Classes}

Let $K$ be a number field, and let $L/K$ be a cyclic extension of prime
degree $\ell$. We will use the following notation:
\begin{itemize}
\item $D_K$, $D_L$ are the groups of nonzero fractional ideals in  
      $K$ and $L$;
\item $H_K$, $H_L$ are the groups of nonzero principal fractional 
      ideals in $K$ and $L$;
\item $\tD_K$ (and $\tH_K$) are the groups of all (resp. of all principal)
      ideals in $L$ generated by ideals (principal ideals) in $K$;
\item $E_K$, $E_L$ are the unit groups in $K$ and $L$;
\item $\Cl(K)$, $\Cl(L)$ denote the ideal class groups in $K$ and $L$;
\item $\sigma$ is a generator of the Galois group $G = \Gal(L/K)$;
\item $e(\fp)$ is the ramification index of the prime $\fp$ in $L/K$;
\item $N$ denotes the relative norm for $L/K$;
\item $A[N] = \{a \in A: Na = 1\}$ is the subgroup of $A$ killed by 
      the norm;
\item $A^G = \{a \in G: a^\sigma = a\}$ is the fix module of the $G$-module $A$; 
\end{itemize}

The following result shows that (\ref{ham}) and (\ref{hams}) are equivalent:

\begin{prop}\label{PAmst}
The difference between ambiguous and strongly ambiguous ideal classes
is measured by the exact sequence
$$ \begin{CD}
 1 @>>> \Am_\st(L/K) @>>> \Am(L/K) @>>> (E_K \cap N L^\times)/NE_L  @>>> 1.
    \end{CD} $$ 
\end{prop}

\begin{proof}
We will first construct the map 
\begin{equation}\label{Enu} 
\nu: \Am(L/K) \lra (E_K \cap N L^\times)/NE_L 
\end{equation}
and then prove exactness by computing the image and the 
kernel of $\nu$. 

An ideal class $[\fa] \in \Cl(L)$ is ambiguous if and only if 
$\fa^\sigma = (\alpha) \fa $ for some $\alpha \in L^\times$.
Taking norms and canceling $N\fa$ yields $(N\alpha) = (1)$, 
which shows that $N\alpha$ is a unit and thus belongs to 
$E_K \cap NL^\times$. 

The map sending $[\fa] \in \Am(L/K)$ to $N\alpha \in E_K \cap NL^\times$ 
is not well-defined since we are allowed to change $\alpha$ by
a unit. Thus we obtain a well-defined homomorphism $\nu$ as in 
(\ref{Enu}).

For showing that $\nu$ is surjective assume 
$\beta \in E_K \cap NL^\times$ and write $\beta = N\alpha$.
Then $(N \alpha) = (1)$, so by Hilbert's Theorem 90 for ideals 
(see Artin's G\"ottingen lectures in \cite[p. 284]{Cohn}; we can 
actually imitate the proof of Hilbert 90 for elements by setting 
$\fa + \fb = \gcd(\fa,\fb)$) we have $(\alpha) = \fa^{\sigma-1}$ for 
some ideal $\fa$, i.e., $\fa^\sigma = (\alpha) \fa$, and this implies 
$[\fa] \in \Am_\st(L/K)$ and $\nu([\fa]) = \beta$.

The kernel of $\nu$ consists of all ideal classes 
$[\fa] \in \Am(L/K)$ for which $\fa^{\sigma-1} = (\alpha)$ and 
$N\alpha = N\eta$ for some unit $\eta \in E_L$. Thus we can 
write $ \fa^\sigma = (\alpha) \fa $ in the form 
$\fa^\sigma = (\alpha/\eta) \fa $. Now $N(\alpha/\eta) = 1$, 
hence $\alpha/\eta = \gamma^{1-\sigma}$ for some
$\gamma \in L^\times$, and we find $\fa^\sigma = \gamma^{1-\sigma} \fa $,
or $(\fa\gamma)^\sigma = \fa\gamma $. The last equation tells
us that $[\fa] = [\fa \gamma]$ is strongly ambiguous. Conversely, every 
strongly ambiguous class is easily seen to be in the kernel of $\nu$, 
and we have shown that $\ker \nu = \Am_\st(L/K)$.
This completes the proof.
\end{proof}

\section{A Group-Theoretical Lemma}

The following simple but useful lemma will be our main new tool in our
version of the proof of Thm.~\ref{main}:

\begin{lem}\label{L2}
Assume that the diagram
$$ \begin{CD}
    1  @>>> A @>>> B @>>> C @>>> 1 \\
    @.  @V{\alpha}VV @V{\beta}VV @V{\gamma}VV  @. \\
    1  @>>> A' @>>> B' @>>> C' @>>> 1    
    \end{CD} $$
is commutative with exact rows, and that the vertical maps are injective.
If two out of the indices $(A':A)$, $(B':B)$ and  $(C':C)$ are finite, 
then all three are, and we have
$$ (B':B) = (A':A)(C':C). $$
\end{lem}

We can drop the assumption that the vertical maps be injective if 
we assume that the vertical maps have finite kernel and cokernel, and 
define the index $(A':A)$ (formally) by
$(A':A) = \frac{\# \coker \alpha}{\# \ker \alpha}$. In our
applications, the vertical maps are always injective. Our
proof below is for the general case:

\begin{proof}[Proof of Lemma~\ref{L2}]
Since the kernel and the cokernel of $\alpha$ are finite, the index
$(A':A) = \frac{\# \coker \alpha}{\# \ker \alpha}$ is a rational number. 
Since the alternating product of the orders of groups in an exact sequence
of finite abelian groups equals $1$, the Snake Lemma gives
$$ \# \ker \alpha \cdot \# \ker \gamma \cdot \# \coker \beta
   = \# \ker \beta \cdot \# \coker \alpha \cdot \# \coker \gamma, $$
or
$$ \frac{\# \coker \alpha}{\# \ker \alpha} \cdot 
   \frac{\# \coker \gamma}{\# \ker \gamma} = 
   \frac{\# \coker \beta}{\# \ker \beta}, $$
which immediately implies our claim. 

The fact that all indices are finite when two of them are is easily
seen to hold by following the proof above.
\end{proof}

\section{Proof of Theorem~\ref{main}}

The basic idea for proving the ambiguous class number formula is
the same as in the classical index calculations of class field theory:
we transform indices of groups involving ideals into indices involving 
principal ideals, and then compute these from indices of unit groups 
and field elements. 

The exact sequence
\begin{equation}\label{ExAst} \begin{CD} 
   1 @>>> H_L^G @>>> D_L^G @>>>  \Am_\st(L/K)  @>>> 1
    \end{CD} \end{equation}
 tells us that 
$$ \# \Am_\st(L/K) = (D_L^G:H_L^G)
   =  \frac{(D_L^G:\tH_K)}{(H_L^G : \tH_K)}, $$
where $\tH_K$ denotes the group of principal ideals in $L$ generated by
elements of $K^\times$. Now 
$$ (D_L^G: \tH_K) 
       = (D_L^G:\tD_K)(\tD_K: \tH_K)
       = h_K \cdot (D_L^G:\tD_K)  $$ 
since $\tD_K/\tH_K \simeq D_K/H_K = \Cl(K)$. The index on the right 
is easily computed:

\begin{lem}\label{LDLG}
Let $L/K$ be a cyclic extension. Then 
$$ (D_L^G:\tD_K) = \prod_{\fp \fin} e(\fp), $$
where the product of over all finite primes $\fp$, and where 
$e(\fp)$ is the ramification index of $\fp$ in $L/K$.
\end{lem}

Since modern introductions to class field theory usually use the
idel theoretic language, some readers may not be familiar with 
this formulation of Lemma~\ref{LDLG}. Here we would like to show 
that the formula can be proved quite easily. References are Artin's 
lectures on class field theory in the appendix of Cohn's book 
\cite[p. 286]{Cohn} and Lang \cite[p. 308]{Lang}. 

We begin our proof with the observation that the group $D_L^G/D_K$ 
is the fix module of $D_L/D_K$; the last group is the direct sum 
(as a $G$-module) of the groups 
$$ I_\fp = \la \fP_1, \ldots, \fP_g \ra / \la \fp \ra, $$
where $\fp \cO_L = (\fP_1 \cdots \fP_g)^{e(\fp)}. $ Since taking
fix modules commutes with direct sums, we find
$$ D_L^G/D_K = \bigoplus_\fp I_\fp^G. $$

Assume now that 
$$ \fP_1^{a_1} \cdots \fP_g^{a_g} \in I_\fp^G. $$
Since $\sigma$ acts transitively on the $\fP_j$, invariance
implies $a_1 = \ldots = a_g =: a$. Thus every ideal in $I_\fp^G$
has the form $(\fP_1 \cdots \fP_g)^a$, and each such ideal is
$G$-invariant. The map $I_\fp \lra \Z/e(\fp)\Z$ defined by sending
$(\fP_1 \cdots \fP_g)^a$ to the residue class $a + e(\fp)\Z$ is
easily seen to be an isomorphism, and now our claim follows:
$$ D_L^G/D_K \simeq \bigoplus_\fp I_\fp^G \simeq \bigoplus_\fp \Z/e(\fp)\Z. $$
This completes the proof of Lemma~\ref{LDLG}.
\medskip

Collecting everything so far we have 
\begin{equation}\label{Am3}
 \# \Am_\st(L/K) = h_K \cdot \frac{\prod_{\fp \fin} e(\fp) }{(H_L^G : \tH_K)}, 
\end{equation}
and so it remains to compute the index in the denominator.
We will do this by reducing it to an index involving numbers 
instead of ideals. To this end we introduce the group
$$ \Delta = \{\alpha \in L^\times: \alpha^{1-\sigma} \in E_L\}. $$
Applying Lemma~\ref{L2} to the diagram
$$ \begin{CD}
    1  @>>> E_L @>>> K^\times E_L   @>>>  \tH_K @>>> 1 \\
   @.      @VVV @VVV @VVV @. \\
    1 @>>>  E_L @>>> \Delta @>>> H_L^G @>>> 1
   \end{CD} $$
where the maps onto the groups of principal ideals is the one 
sending elements to the ideals they generate, shows that 
\begin{equation}\label{EDel1}
        (H_L^G : H_K) = (\Delta : K^\times E_L). 
\end{equation}
Applying Lemma~\ref{L2} to the commutative diagram
$$ \begin{CD}
  1 @>>> K^\times @>>> K^\times E_L   @>{1-\sigma}>>  E_L^{1-\sigma} @>>> 1 \\
   @.      @VVV @VVV @VVV @. \\
  1 @>>>  K^\times @>>> \Delta @>{1-\sigma}>> E_L[N]  @>>> 1
   \end{CD} $$
shows that
\begin{equation}\label{EDel2}  
(\Delta:K^\times E_L) = (E_L [N]: E_L^{1-\sigma}). 
\end{equation}

Observe that $E_L \cap (L^\times)^{1-\sigma} = E_L[N]$: a unit
killed by the norm is an element of $(L^\times)^{1-\sigma}$ by
Hilbert's Theorem 90, and the converse is trivial.

Combining (\ref{EDel1}) and  (\ref{EDel2}) shows that
$$ (H_L^G : H_K) = (E_L [N]: E_L^{1-\sigma}), $$
and plugging this into (\ref{Am3}) we obtain
$$ \# \Am_\st(L/K) = h_K  \cdot
      \frac{\prod_{\fp \fin} e(\fp)}{(E_L[N] : E_L^{1-\sigma})}. $$
Determining the index in the denominator is a nontrivial task,
and the result, in the classical literature, is called the

\begin{thm}[Unit Principal Genus Theorem]\label{TPGU}
For cyclic extensions $L/K$ we have 
$$ \frac{(E_K : NE_L)}{(E_L[N] : E_L^{1-\sigma})} 
                              = \frac1{(L:K)} \prod_{\fp \mid \infty} e(\fp). $$
\end{thm}

Theorem~\ref{TPGU} is a well known result from class field
theory (it follows painlessly from the unit theorem of Herbrand and
Artin using the Herbrand index machinery; see e.g. 
Lang \cite[p.~192, Cor. 2]{LANT} or Childress \cite[Prop. 5.10]{Chil}.

Plugging this into our expression for $\Am_\st(L/K)$ we get the desired
formula, and our proof is complete.

\section*{Acknowledgements}
I thank the referee for the careful reading of the manuscript and
for many helpful suggestions.

\end{document}